\documentclass[11pt,reqno]{amsart}
\usepackage{amsmath,amsfonts,amssymb,bbm}

\usepackage{a4wide}
\usepackage{xfrac}
\usepackage{verbatim}
\usepackage[shortlabels]{enumitem}

\usepackage[unicode,hypertexnames=false,colorlinks=true,linkcolor=blue,citecolor=blue]{hyperref}
\usepackage{anyfontsize}
\usepackage{hyperref}
\usepackage[capitalise, nameinlink, noabbrev]{cleveref} 
\hypersetup{colorlinks=true,
	linkcolor=blue,
	citecolor=red,
}

\hypersetup{colorlinks,breaklinks,
	linkcolor=[rgb]{0,0,0},
	citecolor=[rgb]{0,0,0},
	urlcolor=[rgb]{0,0,0}}

\newcommand{\be}{\begin{equation}}
\newcommand{\ee}{\end{equation}}
\theoremstyle{plain}

\newtheorem{theorem}{Theorem}[section]
\newtheorem{proposition}[theorem]{Proposition}
\newtheorem{corollary}[theorem]{Corollary}
\newtheorem{lemma}[theorem]{Lemma}

\usepackage{mathtools}

\newcommand{\R}{\mathbb{R}}
\newcommand{\vf}{\varphi}

\newcommand{\N}{\mathbb{N}}
\newcommand{\eps}{\varepsilon}

\newcommand{\HH}{\mathcal{H}}

\newcommand{\bea}{\begin{equation*}\begin{aligned}}
		\newcommand{\eea}{\end{aligned}\end{equation*}}

\numberwithin{equation}{section}

\title{An epiperimetric inequality for odd frequencies in the thin obstacle problem}

\author[M. Carducci]{Matteo Carducci}\thanks{}
	\address {Matteo Carducci \newline \indent
		Classe di Scienze, Scuola Normale Superiore \newline \indent
		Piazza dei Cavalieri 7, 56126 Pisa - ITALY}
	\email{\href{mailto:matteo.carducci@sns.it}{matteo.carducci@sns.it}}
	
	\author[B. Velichkov]{Bozhidar Velichkov}\thanks{}
	\address {Bozhidar Velichkov \newline \indent
		Dipartimento di Matematica, Universit\`a di Pisa \newline \indent
		Largo B. Pontecorvo 5, 56127 Pisa - ITALY}
	\email{\href{mailto:bozhidar.velichkov@unipi.it}{bozhidar.velichkov@unipi.it}}

\begin{document}
	\begin{abstract}
 We prove for the first time an epiperimetric inequality for the thin obstacle Weiss' energy with odd frequencies and we apply it to solutions to the thin obstacle problem with general $C^{k,\gamma}$ obstacle. In particular, we obtain the rate of convergence of the blow-up sequences at points of odd frequencies and the regularity of the strata of the corresponding contact set. We also recover the frequency gap for odd frequencies obtained by Savin and Yu.
	\end{abstract}
	
	\maketitle
	
	\tableofcontents
	\section{Introduction}
We consider solutions $u:B_1\subset\R^{n+1}\to \R$ to the thin obstacle problem 
\be\label{def:sol-con-ostacolo}
\begin{cases}
    \Delta u=0 & \text{in } B_1\setminus \{u(x',0)=\vf(x')\},\\
    \Delta u \le0 & \text{in } B_1,\\
    u(x',0)\ge\vf(x') &\text{on } B_1':=B_1\cap\{x_{n+1}=0\},\\
    u(x',x_{n+1})=u(x',-x_{n+1})& \text{in } B_1,
\end{cases}
\ee
with obstacle $\vf:B_1'\subset\R^{n}\to \R$ satisfying 
    \be\label{e:hypo-phi}\vf\in C^{k,\gamma}(B_1'),\quad\text{with}\quad k\in\N_{\ge2}\cup\{+\infty\}\quad\text{and}\quad\gamma\in(0,1).\ee
The thin obstacle problem can also be formulated as a variational problem 
$$\min_{w\in H^1(B_1)}\left\{\int_{B_1}|\nabla w|^2\,dx: w\ge \vf \text{ on } B_1',\ w=g \text{ on } \partial B_1,\ w(x,x_{n+1})=w(x,-x_{n+1})\right\},$$ 
for a given boundary datum $g:\partial B_1\to\R$, which is even with respect to $\{x_{n+1}=0\}$, in the sense that 
$g(x',x_{n+1})=g(x',-x_{n+1})$ for every $(x',x_{n+1})\in B_1.$

The optimal regularity of the solution $u$ was obtained in \cite{ac04} where it was shown that $u\in \text{Lip}(B_1)\cap C^{1,\frac12}(B_1^+\cup B_1')$, where $B_1^+$ is the open half-ball $B_1^+:=B_1\cap\{x_{n+1}>0\}$; this regularity is also optimal as there are $3/2$-homogeneous global solutions to \eqref{def:sol-con-ostacolo} with $\varphi=0$.\medskip

In this paper we are interested in the local behavior of $u$ around points on the hyperplane $\{x_{n+1}=0\}$, which is 
determined by the structure of the contact set 
$$\Lambda(u):=\{x'\in B_1': u(x',0)=\varphi(x')\},$$ 
and its free boundary $\Gamma(u)$ defined as the topological boundary of $\Lambda(u)$ with respect to the relative topology of the hyperplane $\{x_{n+1}=0\}$:
$$\Gamma(u):=\partial\Lambda(u)\subset\Lambda(u).$$ 

 Since $\vf$ satisfies the regularity assumption \eqref{e:hypo-phi}, we can reduce the thin-obstacle problem \eqref{def:sol-con-ostacolo} to a thin-obstacle problem with right-hand side and a zero obstacle by proceeding as in \cite{gr19}, \cite{gp09}, \cite{css08}, and \cite{bfr18}. Given $x_0\in\Lambda(u)$, let $q_k^{(x_0)}(x')$ be the $k$-th Taylor polynomial of $\vf$ at $x_0\in\Gamma(u)$ and $\widetilde q_k^{(x_0)}(x)$ be the polynomial of degree $k$ which is the harmonic extension of $q_k^{(x_0)}(x')$. Then, the function
  \bea 
  v(x)=u^{(x_0)}(x):=u(x)-\vf(x')+q_k^{(x_0)}(x')-\widetilde q_k^{(x_0)}(x),
  \eea 
  solves the following problem 
  \be\begin{cases}\label{def:v}
		\Delta  v(x)=h(x) & \mbox{in } B_1\setminus \{v(x',0)=0\}\\
		\Delta  v(x)\le h(x) & \mbox{in } B_1,\\
    v(x',0)\ge0 & \mbox{on } B_1',\\
  v(x',x_{n+1})=v(x',-x_{n+1}) &\mbox{in } B_1,
	\end{cases}\ee
 where $h(x):=- \Delta_{x'}(\vf(x')-q_k^{(x_0)}(x'))$. In particular 
 $$|h(x)|\le C|x-x_0|^{k+\gamma-2}\quad\text{for every}\quad x\in B_1,$$ for some constant $C>0$, depending only on $n$, $\varphi$, $k$ and $\gamma$.
 
 As in \cite{css08, gp09,bfr18,gr19}, we consider the following truncated Almgren's frequency function 
 \bea\Phi^{x_0}(r,v):=(r+C_{\Phi}r^{1+\theta})\frac{d}{dr}\log\max\{H^{x_0}(r,v),r^{n+2(k+\gamma-\theta)}\},\eea for $\theta\in(0,\gamma)$, $C_{\Phi}>0$ large enough and 
 $$H^{x_0}(r,v):=\int_{\partial B_r(x_0)}v^2\,d\HH^n.$$
 The function $r\mapsto \Phi^{x_0}(r,v)$ is monotone increasing for $r>0$ small enough (see \cite{gr19} and \cref{prop:mono}) and, as above, we define the frequency at a point $x_0\in\Lambda(u)$ as 
 $$\Phi^{x_0}(0^+,v):=\lim_{r\to0^+}\Phi^{x_0}(r,v)\quad\text{where}\quad v=u^{(x_0)}.$$
The above monotonicity formula  allows to decompose the contact set and the free boundary into sets of points with the same frequency. Precisely, given a frequency $\mu>0$, lying below the threshold $k+\gamma$ determined by the obstacle $\varphi$, we consider the following subsets of  the contact set $\Lambda(u)$ and its boundary $\Gamma(u)$: 
$$\Gamma_\mu(u):=\{x_0\in\Gamma(u): \Phi^{x_0}(0^+,v)=n+2\mu\}\quad\text{for every}\quad \mu<k+\gamma,$$
 and \bea \Lambda_\mu(u):=\{x_0\in\Lambda(u): \Phi^{x_0}(0^+,v)=n+2\mu\}\quad\text{for every}\quad \mu<k+\gamma.\eea
 Given $\mu<k+\gamma$, $x_0\in\Lambda_{\mu}(u)$ and $v=u^{(x_0)}$ as in \eqref{def:v}, consider the rescalings
 $$\widetilde v_{x_0,r}(x)=\frac{v(x_0+r\cdot)}{\|v(x_0+r\cdot)\|_{L^2(\partial B_1)}}.$$
 Thanks to the (almost-)monotonicity of $r\mapsto \Phi^{x_0}(r,v)$ and the minimality of $u$, we have that every sequence $r_n\to0$ admits a subsequence (still denoted by $r_n$) such that $\widetilde v_{r_n,x_0}$ converges strongly in $H^1(B_1)$ to some $v_0$. 
%
 %
 Moreover, every $v_0$ obtained this way is $\mu$-homogeneous and solves the following thin-obstacle problem, which is precisely \eqref{def:sol-con-ostacolo} with $\varphi=0$: 
\be\label{def:sol}
\begin{cases}
    \Delta u=0 & \text{in } B_1\setminus \{u(x,0)=0\},\\
    \Delta u\le0 & \text{in } B_1,\\
    u(x',0)\ge0 &\text{on } B_1',\\
    u(x,x_{n+1})=u(x,-x_{n+1})& \text{in } B_1.
\end{cases}
\ee

\subsection{Admissible frequencies: state of the art} We say that $\mu$ is {\it an admissible frequency in $\R^{n+1}$} if there exists a non-trivial $\mu$-homogeneous solution to \eqref{def:sol} in $B_1\subset \R^{n+1}$, and we indicate the set of admissible frequencies by
\be\label{eq:a}\mathcal{A}_n:=\{\mu>0: \text{there is a $\mu$-homogeneous solution to \eqref{def:sol} in $\R^{n+1}$}\}.\ee
We notice that, since every $\mu$-homogeneous solution in $\R^{n+1}$ can be extended to a $\mu$-homogeneous solution in $\R^{n+2}$, we have the inclusions $\mathcal A_{n}\subset\mathcal A_{n+1}$ for every $n\ge 1$.

In dimension $n+1=2$, it is known (see for instance \cite{psu12} and the references therein) that the set of admissible frequencies is given by
$$\mathcal A_1=\{2m-1/2\ :\ m\ge 1\}\cup\{2m\ :\ m\ge1 \}\cup\{2m+1\ :\ m\ge 0\}.$$
The known results (up to this point) about the admissible set $\mathcal A_n$ in dimension $n>1$ are the following: 
 \begin{itemize}
     \item in \cite{acs08} (see also \cite{gps16,fs16,csv20,car23,car24} for an approach based on epiperimetric inequalities) it was shown that $$\mathcal{A}_n\cap \Big((0,1)\cup\left(1,\sfrac32\right)\cup\left(\sfrac32,2\Big)\right)=\emptyset\,;$$ 
     \item in \cite{csv20} it was proved via an  epiperimetric inequality that, for every $m\in\N$, there are constants $c^\pm_{n,m}>0$, depending only on $n$ and $m$, such that 
     $$\mathcal{A}_n\cap \Big((2m-c_{n,m}^-,2m)\cup(2m,2m+c_{n,m}^+)\Big)=\emptyset\,;$$ 
     we refer also to \cite{sy22} where this result was obtained via different arguments;
     \item in \cite{sy22} it was shown that, for every $m\in\N$ there are constants $c^\pm_{n,m}>0$, depending only on $n$ and $m$, such that 
     \bea \mathcal{A}_n\cap \Big((2m+1-c_{n,m}^-,2m+1)\cup(2m+1,2m+1+c_{n,m}^+)\Big)=\emptyset\,;
     \eea
     \item in the recent paper \cite{FranceschiniSavin2024}
 it was shown that 
 \bea \mathcal{A}_n\cap (2m,2m+1)=\emptyset\quad\text{for all}\quad m\in\N\,;
     \eea  
 \item finally, we notice that it is currently an open question whether $\mathcal A_n\setminus\mathcal A_1=\emptyset$.
 \end{itemize}
\subsection{Regularity of the free boundary: state of the art} 
For what concerns the regularity of the free boundary $\Gamma(u)$ and the contact set $\Lambda(u)$ of a solution $u$ to \eqref{def:sol-con-ostacolo}, with obstacle $\vf$ satisfying \eqref{e:hypo-phi},
the known results are the following.
In the lowest dimension $n+1=2$
$$\bigcup_{0<\mu< k+\gamma}\Gamma_\mu(u)=\bigcup_{0<\mu< k+\gamma}\{\Gamma_\mu(u): \ \mu\in 2\N\cup (2\N-1/2)\}$$  is a discrete set, while for the contact set we have $$\bigcup_{0<\mu< k+\gamma}\Lambda_\mu(u)\setminus\Gamma_\mu(u)=\bigcup_{0<\mu< k+\gamma}\{\Gamma_{2m+1}(u): \ m\in\N_{\ge0}\}.$$
In dimension $n+1>2$, when the obstacle $\varphi$ is zero (or analytic), it holds
$$\Gamma(u)=\bigcup_{0<\mu< +\infty}\Gamma_\mu(u)\qquad\text{and}\qquad \Lambda(u)=\bigcup_{0<\mu< +\infty}\Lambda_\mu(u).$$
In the case of a general obstacle $\varphi$, we only consider frequencies $\mu$ below the threshold $k+\gamma$ determined by the regularity of $\varphi$, so we have $$\Gamma(u)\supset\bigcup_{0<\mu< k+\gamma}\Gamma_\mu(u)\qquad\text{and}\qquad \Lambda(u)\supset\bigcup_{0<\mu< k+\gamma}\Lambda_\mu(u).$$
Here below, we briefly recall some of the known regularity results in the literature up to this point; for more detailed introduction to the topic we refer to the book \cite{psu12} and to the surveys \cite{survey,dasa18}.

 \begin{itemize}
 \item \textit{Points of frequency one.} The points of frequency $1$ lie in the interior of the contact set, that is, $\Lambda_1(u)$ is an open subset of $\R^n$, while $\Gamma_1(u)=\emptyset$. \smallskip
\item \textit{Regular points.} The contact points of frequency $3/2$, which are called {\it regular points}, are contained in the free boundary $\Gamma(u)$, that is $\Lambda_{3/2}(u)=\Gamma_{3/2}(u)$. Moreover, $\Gamma_{3/2}(u)$ is an open subset of $\Gamma(u)$ and a $C^{1,\alpha}$-regular $(n-1)$-dimensional manifold; this was proved in \cite{acs08} in the case $\vf\equiv0$ and in \cite{css08} in the case $\vf\not\equiv0$;
see also \cite{gps16,fs16,gpps17,csv20} for proofs based on epiperimetric inequalities. The $C^{\infty}$ regularity of $\Gamma_{3/2}(u)$ was obtained in \cite{kps15,ds16} in the case of zero obstacle. For generic boundary data, in \cite{fr21} it was shown that the non-regular part of the free boundary is at most $(n-2)$-dimensional (for $C^\infty$ obstacle), while in \cite{ft23} it was proved that the non-regular set has zero $\mathcal{H}^{n-3-\alpha}$ measure (for zero obstacle). In particular, for $n+1\le 4$, the free boundary is {generically} smooth (for zero obstacle).
\smallskip 
     \item \textit{Singular points.} For every $m\in\N_{\ge1}$ with $2m\le k$, the contact points of frequency $2m$ (the so-called \textit{singular points}) are contained in the free boundary $\Gamma(u)$, that is, $\Lambda_{2m}(u)=\Gamma_{2m}(u)$. Moreover, each of the sets $\Gamma_{2m}(u)$ is contained in a countable union of $(n-1)$-dimensional $C^1$ manifolds. This result was proved in \cite{gp09} in the case $\vf\equiv0$ and $\vf\in C^{k,1}$ and in \cite{gr19} in the case $\vf\in C^{k,\gamma}$.
     The same result with a logarithmic modulus of continuity was obtained via log-epiperimetric inequality in \cite{csv20} for zero obstacle and in \cite{car24} in the general case $\vf\in C^{k,\gamma}$. Moreover, in \cite{fj21} it was proved that each stratum of the
     singular set is locally contained in a single $C^2$ manifold, up to a lower dimensional subset, in the case $\vf\equiv0$.
     \smallskip
     \item \textit{Points of odd frequency.} In the case $\vf\equiv0$, in \cite{sy23} it was shown that, for every $m\ge 0$, the set $\Lambda_{2m+1}(u)$ is contained in a countable union of $(n-1)$-dimensional manifolds of class $C^{1,\alpha}$. Contrary to what happens for points of frequency $3/2$ and $2m$, the points of odd frequency may also lie in the interior of the contact set. In fact, it was shown in \cite{frs20} that for all homogeneous solutions with zero obstacle $\Lambda(u)\equiv\{x_{n+1}=0\}$. 
     It is currently not known if one can find a solution $u$ to \eqref{def:sol} for which the set $\Gamma_{2m+1}(u)$ is not empty.
     We also stress that no epiperimetric inequality for odd frequencies was known until now.\smallskip
    \item \textit{Points of frequency $2m-1/2$.} The last class of points with 2D blow-ups (that is, points at which there are blow-ups depending only on two of the $n+1$ variables) are the points of homogeneity $2m-1/2$ with $m>1$. In dimension $n+1=3$, Savin and Yu \cite{sy22:7/2} proved a regularity result for $\Gamma_{7/2}(u)$ around points at which $u$ admits \textit{half-space} blow-ups; precisely, they showed that this set is the union of a locally discrete set and a set which is locally covered by a $C^{1,\log}$ curve. A general regularity result, up to codimension $3$, about the free boundary points of frequency $2m-1/2$ was proved by Franceschini 
 and Serra in \cite{fs24}.\smallskip
     \item \textit{Rectifiability of the free boundary.} Finally, we notice that in  \cite{fs18} and \cite{fs22} it was shown that the free boundary $\Gamma(u)$ is an $(n-1)$-rectifiable set for any $n+1\ge 2$. In particular, this implies that at $\HH^{n-1}$-almost every point $x_0\in\Gamma(u)$ the blow-up is unique (this was shown in \cite{csv21} in the case $\varphi\equiv0$). This results was improved in \cite{fs24}, where the authors proved that, in the case $\vf\equiv0$, the free boundary $\Gamma(u)$ is covered by countably many $(n-1)$-dimensional manifolds of class $C^{1,1}$, up to a set of Hausdorff dimension $n-2$.
     
 \end{itemize}
  
 \subsection{Main results}\label{sub:intro-main-result} 
 The main result of this paper is an epiperimetric inequality for the Weiss' energy associated to the odd frequencies. Before we state our main theorem, we introduce some notations. 
 For every $m\in\N$, we define the set
 \begin{align*}
 \mathcal{P}_{2m+1}:=\{p: \ & \Delta p=0\ \text{ in } \{x_{n+1}\neq 0\},\ \Delta p\le0\ \text{ in } \R^{n+1},\\
 &\nabla p\cdot x=(2m+1)p,\ p\equiv 0\ \text{ on } B_1',\ p(x',x_{n+1})=p(x',-x_{n+1})\},
\end{align*} 
and we recall that, by \cite{frs20}, the set of admissible blow-ups at any point of frequency $2m+1$ is precisely given by $\mathcal P_{2m+1}$.
Every $p\in\mathcal{P}_{2m+1}$ can be written in the form 
$$p(x',x_{n+1})=-|x_{n+1}|(p_0(x')+x_{n+1}^2p_1(x',x_{n+1}))$$ 
for some homogeneous polynomials $p_0$ and $p_1$ satisfying the inequality $p_0\ge0$ (which follows from the fact that $p$ is superharmonic). We define the operator $T$ as 
 \begin{equation}\label{e:definition-of-T}
 T:\mathcal{P}_{2m+1}\to L^2(B_1')\ ,\qquad p\mapsto T[p]:=p_0.
 \end{equation}
We denote by $W_\mu$ the Weiss' energy associated to the frequency $\mu$, precisely:
\be\label{weiss}W_\mu(u):=\int_{B_1}|\nabla u|^2\,dx-\mu\int_{\partial B_1}u^2\,d\HH^n.\ee 
Our main result is the following epiperimetric inequality for the Weiss' energy $W_{2m+1}$.
 \begin{theorem}[Epiperimetric inequality for $W_{2m+1}$]  \label{thm:epi} 
     There are constants $ \eps>0$, $\delta>0$ and $\kappa>0$, depending only on $n$ and $m$, such that the following holds.
     Let $c\in H^1(\partial B_1)$ be a trace which is even with respect to $\{x_{n+1}=0\}$ and such that $c\ge 0$ on $B_1'$. Let $z(r,\theta)=r^{2m+1}c(\theta)$ be the $(2m+1)$-homogeneous extension of $c$ in $\R^{n+1}$.  
     
     Suppose that there is $p\in\mathcal{P}_{2m+1}$ with $\|p\|_{L^2(\partial B_1)}=1$ such that
     \be\label{eq:vicinanza}
     \|c-p\|_{L^2(\partial B_1)}\le \eps,
     \ee
     and 
     \be\label{eq:quasizero} c\equiv 0 \quad\text{on}\quad \mathcal{Z}_\delta:=\{T[p]\ge\delta\}\cap \partial B_1', \ee 
     where 
     $T$ is the operator from \eqref{e:definition-of-T}.
     Then, there is a function $\zeta\in H^1(B_1)$ satisfying the epiperimetric inequality
     \be\label{eq:epi}
     W_{2m+1}(\zeta)\le (1-\kappa)W_{2m+1}(z),
     \ee 
     and such that $\zeta\ge 0$ on $B_1'$, $\zeta=c$ on $\partial B_1$ and $\zeta$ is even with respect to $\{x_{n+1}=0\}$.
 \end{theorem}
The epiperimetric inequality was first introduced in the 60s by Reifenberg  (\cite{reifenberg}) in the context of minimal surfaces. More recently,  epiperimetric inequalities were proved for different free boundary problems (see e.g. \cite{wei99,gps16,fs16,gpps17,csv18,sv19,csv20,sv21,esv23,car23,car24,ov24}) and were used to deduce regularity results in different contexts.
\smallskip

The epiperimetric inequality from \cref{thm:epi}, together with the ones proved in \cite{csv20,car24} for the energy $W_{2m}$, provide a unified approach for the study of the integer frequencies in the thin obstacle problem, even for general obstacle $\vf\not\equiv0$. We stress that, there are two major differences between \eqref{eq:epi} and the epiperimetric inequalities from \cite{csv20,car24}. First, contrary to the log-epiperimetric inequalities from \cite{csv20,car24}, \eqref{eq:epi} provides a polynomial decay of the blow-up sequence. Second, in order to apply \cref{thm:epi}, one needs to verify that the closeness conditions \eqref{eq:vicinanza} and \eqref{eq:quasizero} remain valid along blow-up sequences; in \cref{prop:every-rescaled-fundamental} we show that these conditions are self-propagating, that is, they can be deduced from the epiperimetric inequality itself. 
Finally, we notice that in the forthcoming \cite{cc24} (a suitable generalization of) the epiperimetric inequality from \cref{thm:epi} will be one of the ingredients in the proof of a generic regularity result for solutions to the obstacle problem for the fractional Laplacian.\smallskip

We next apply the epiperimetric inequality to solutions $u$ to \eqref{def:sol-con-ostacolo} with general obstacle $\vf$ satisfying \eqref{e:hypo-phi}. 
First, as a consequence of \cref{thm:epi}, we obtain the uniqueness of the blow-up limits with a rate of the convergence. 
\begin{theorem}[Uniqueness of the blow-up limits and rate of convergence of the blow-up sequences]\label{prop:rate} Let $u$ be a solution to the thin obstacle problem \eqref{def:sol-con-ostacolo} with a $C^{k,\gamma}$-regular obstacle $\vf$ satisfying \eqref{e:hypo-phi}.
Let $2m+1\le k$, $ 0\in\Lambda_{2m+1}(u)$, and $v=u^{(0)}$ be given by \eqref{def:v}.

If $v\not\equiv0$,  $\|v\|_{L^2(\partial B_1)}\le1$ and \be\label{def:rescalings22}v_r(x):=\frac{v(rx)}{r^{2m+1}},
\ee
then there are a non-zero $p\in\mathcal P_{2m+1}$ and $\rho>0$ small enough such that $$\|v_r-p\|_{L^\infty(B_1)}\le Cr^{\alpha}\quad\text{for every}\quad r\in(0,\rho),$$ where  $\alpha>0$ depends only on $n$ and $m$ and $C>0$ depends only on $n$, $m$, $\varphi$, $k$ and $\gamma$.
\end{theorem}
The proof of \cref{prop:rate} follows from \cref{thm:epi} and \cref{prop:every-rescaled-fundamental} (see \cref{section5}). We notice that \cref{prop:rate} can also be obtained by combining \cref{thm:epi} with the uniqueness and the non-degeneracy results from \cite{frs20} for points of frequency $2m+1$, which guarantee that the closeness assumptions \eqref{eq:vicinanza} and \eqref{eq:quasizero} remain satisfied at every scale.\smallskip 

As a consequence of \cref{prop:rate} we obtain that for $2m+1\le k$, the $j$-strata of $\Lambda_{2m+1}(u)$ are contained in $C^{1,\alpha}$ manifolds of dimension $j$, for every $j=1,\ldots,n-1$. In the case $\vf\equiv0$, this stratification result was already established by Savin and Yu in \cite{sy23} via an improvement of flatness technique. 
\begin{corollary}[Stratification and rectifiability of the contact set $\Lambda_{2m+1}(u)$]\label{cor:stratification}
Let $u$ be a solution to the thin obstacle problem \eqref{def:sol-con-ostacolo}, with a $C^{k,\gamma}$-regular obstacle $\vf$ satisfying \eqref{e:hypo-phi}.
Then, for every $m\in\N$ such that $2m+1\le k$, the set $\Lambda_{2m+1}(u)$ is contained in the union of countably many manifolds of class $C^{1,\alpha}$ for some $\alpha>0$.
More precisely $$\Lambda_{2m+1}(u)=\bigcup_{j=0}^{n-1}\Lambda_{2m+1}^j(u)$$ where, for every $j=0,\ldots,n-1$ and every $x_0\in\Lambda_{2m+1}^j(u)$, there is a neighborhood $\mathcal U_{x_0}$ such that $\mathcal U_{x_0}\cap \Lambda_{2m+1}^j(u)$ is contained in a $j$-dimensional manifold of class $C^{1,\alpha}$.
\end{corollary}
We also use the epiperimetric inequality in \cref{thm:epi} and an epiperimetric inequality for negative energies $W_{2m+1}$ (see \cref{prop:epi-negative}) to give another proof of the frequency gap around the odd frequencies, which was first obtained in \cite{sy22}.
\begin{theorem}[Frequency gap]\label{prop:gap} 
   Let $\mathcal{A}_n$ as in \eqref{eq:a}, then
$$\mathcal{A}_n\cap \left((2m+1-c_{n,m}^-,2m+1)\cup(2m+1,2m+1+c_{n,m}^+)\right)=\emptyset,$$ 
     for some constants $c^\pm_{n,m}>0$, depending only on $n$ and $m$. 
\end{theorem}
Even in this case, with the analogous result for even frequencies in \cite{csv20}, we get a unified epiperimetric inequality approach for the frequency gap around integer frequencies.\smallskip

Soon after the present paper was published as a preprint, Franceschini and Savin proved in \cite{FranceschiniSavin2024} that there no admissible frequencies in the intervals of the form $(2m,2m+1)$. This improves the lower bound in \cref{prop:gap} to $c_{n,m}^-=1$, which is also optimal.

\subsection{Plan of the paper}\label{sub:plan-of-the-paper}  In \cref{section2} we recall the truncated Almgren's frequency function, the blow-ups and the Weiss' energy for solutions with obstacle $\vf\not \equiv0$. 

In \cref{section3} we prove the epiperimetric inequality for $W_{2m+1}$, i.e.~\cref{thm:epi}. The strategy is to decompose the trace using the eigenfunctions of spherical Laplacian $\Delta_{\mathbb{S}^n}$. 
We will use the eigenfunctions of the half-sphere for the lower modes and the eigenfunctions which are $0$ on the set $\mathcal{Z}_\delta$ (defined in \eqref{eq:quasizero}) for the higher modes.
We construct this decomposition by using the implicit function theorem (see \cref{lemma:implicit}). We then define the competitor $\zeta$ by increasing the homogeneity of the higher modes and we prove that $\zeta$ satisfies the epiperimetric inequality \eqref{eq:epi} by using \cref{lemma:lower}, \cref{lemma:higher} and \cref{lemma:double-product}.

In \cref{section4} we prove that we can apply the epiperimetric inequality to solutions of \eqref{def:sol} at every scale. The point here is that the epiperimetric inequality provides a control on the oscillation of the rescalings $v_r(x)=r^{-(2m+1)}v(rx)$, while on the other hand the information about the frequency at the point $x_0=0$ is contained in the rescalings $\widetilde v_\rho(x)=\|v(\rho\cdot)\|_{L^2(\partial B_1)}^{-1}v(\rho x)$, which converge to homogeneous solutions of unit $L^2(\partial B_1)$ norm. In order to apply the epiperimetric inequality at every scale we need to show that the conditions \eqref{eq:vicinanza} and \eqref{eq:quasizero} are satisfied at every scale. 

We choose $\rho$ small enough such that $\widetilde v_\rho$ is close to a homogeneous global solution. Then, we consider the double rescalings $(\widetilde v_\rho)_r\in H^1(B_1)$ defined in  \eqref{eq:doublerescalings}. Using the oscillation control provided by the epiperimetric inequality, we show that the traces $(\widetilde v_\rho)_r|_{\partial B_1}$ satisfy the conditions \eqref{eq:vicinanza} and \eqref{eq:quasizero}, so we can apply the epiperimetric inequality to $(\widetilde v_\rho)_r|_{\partial B_1}$ for all $r\in(0,1)$.

In \cref{section5} we prove the uniqueness of blow-up limit (\cref{prop:rate}) and the stratification of the contact set (\cref{cor:stratification}). We notice that, once we know that we can apply \cref{thm:epi} at every scale, these results are a standard application of the epiperimetric inequality.

Finally, in \cref{section6} we prove an epiperimetric inequality for negative energies $W_{2m+1}$ (see \cref{prop:epi-negative}) and use it, together with \cref{thm:epi}, to obtain the frequency gap in \cref{prop:gap}.
 \subsection{Notations} \label{notations}
Given $x\in\R^{n+1}$, we write $x=(x',x_{n+1})$, with $x'\in\R^n$ and $x_{n+1}\in\R$.\\
For any set $A\subset\R^{n+1}$, we will use the notation 
$$A^+:=A\cap\{x_{n+1}>0\}\quad\text{and}\quad A':=A\cap\{x_{n+1}=0\}.$$
We will write $m\in\N_{\ge j}$ if $m$ is an integer and $m\ge j$. From now, by $m$ we will denote only integers in $\N_{\ge 0}$. 
 \subsection*{Acknowledgement} The authors are supported by the European Research Council (ERC), through the European Union's Horizon 2020 project ERC VAREG - \it Variational approach to the regularity of the free boundaries \rm (grant agreement No. 853404); they also acknowledge the MIUR Excellence Department Project awarded to the Department of Mathematics, University of Pisa, CUP I57G22000700001. B.V. acknowledges also support from the projects PRA 2022 14 GeoDom (PRA 2022 - Università di Pisa) and MUR-PRIN ``NO3'' (No. 2022R537CS).
\section{Preliminaries}\label{section2}
\subsection{Almgren's frequency function and blow-ups}
We recall the following two propositions from \cite{gr19}. 
\begin{proposition}[Truncated Almgren's frequency function]\label{prop:mono} Let $u$ be a solution to the thin obstacle problem \eqref{def:sol-con-ostacolo}, with obstacle $\vf$ satisfying \eqref{e:hypo-phi}. Let $v=u^{(x_0)}$ given by \eqref{def:v} with $x_0\in\Lambda(u)$. Let $\theta\in(0,\gamma)$, we define \bea\Phi^{x_0}(r,v):=(r+C_{\Phi}r^{1+\theta})\frac{d}{dr}\log\max\{H^{x_0}(r,v),r^{n+2(k+\gamma-\theta)}\},\eea and $$H^{x_0}(r,v):=\int_{\partial B_r(x_0)}v^2\,d\HH^n\quad\text{and}\quad \mathcal{I}^{x_0}(r,v):=\int_{\partial B_r(x_0)}v\partial_{\nu}v\,d\HH^n.$$
 We drop the dependence on $x_0$ if $x_0=0$.
 
 If $C_{\Phi}>0$ is large enough, then there is $r_0>0$ such that $$r\mapsto \Phi^{x_0}(r,v)\quad\text{is non-decreasing for every $r\in(0,r_0)$.}$$
	Moreover if $x_0\in\Lambda_\mu(u)$, with $\mu<k+\gamma$, then, for every $\eps>0$ \bea r\mapsto \frac{H^{x_0}(r,v)}{r^{n+2\mu}}\quad\text{is non-decreasing for every $r\in(0,r_0)$},\eea \bea r\mapsto \frac{H^{x_0}(r,v)}{r^{n+2\mu+\eps}}\quad\text{is non-increasing for every $r\in(0,r_\eps)$}\eea and $$\phi^{x_0}(r,v)=(1+C_{\Phi}r^{\theta})\left(n+2r\frac{\mathcal{I}^{x_0}(r,v)}{H^{x_0}(r,v)}\right)\quad\text{for every}\quad r\in(0,r_0).$$
 In particular, the rescalings \bea\label{rescaling}\widetilde v_{x_0,r}(x):=\frac{v(x_0+rx)}{\|v(x_0+r\cdot)\|_{L^2(\partial B_1)}}\eea converge in $C^{1,\alpha}(B_1^+)$, as $r\to0^+$, up to subsequences, to some function $v_{0}$ which is a solution to the thin obstacle problem \eqref{def:sol} and it is $\mu$-homogeneous.
	\end{proposition}
\begin{proposition}\label{prop:blow-up-homo2}
Let $u$ be a solution to the thin obstacle problem \eqref{def:sol-con-ostacolo}, with obstacle $\vf$ satisfying \eqref{e:hypo-phi}. Let $v=u^{(0)}$ given by \eqref{def:v}. Suppose that, for some $\rho_0>0$, we have that $$H(2,v_{r})\le H_0\quad\text{and}
\quad \phi(2r,v)\le \phi_0\quad\text{for every}\quad r\in(0,\rho_0),$$ where
			 \be\label{def:rescalings2}v_{r}(x):=\frac{v(rx)}{r^\mu},\quad r\in(0,\rho_0].\ee
 Then $$\|v_r\|_{C^{1,\frac12}(B_{3/2}^+)}\le C\quad\text{for every}\quad r\in(0,\rho_0]$$ and for some constant $C>0$, depending only on $H_0$, $\phi_0$, $n$, $\mu$, $\varphi$, $k$, $\gamma$. The same inequality holds if we replace $v$ with $\widetilde v_\rho$, the rescalings in \cref{prop:mono}.
\end{proposition}
 \subsection{Weiss' energy $\widetilde W_\mu$}
Let $u$ be a solution to \eqref{def:sol-con-ostacolo}, with obstacle $\vf$ satisfying \eqref{e:hypo-phi}. Let $v=u^{(0)}$ given by \eqref{def:v} with $0\in\Lambda_\mu(u)$, we consider the following Weiss' energy for the problem \eqref{def:v} with right hand side
 \bea \widetilde W_\mu(v):=W_\mu(v)+\int_{B_1}vh\,dx,\eea where $W_\mu$ is the Weiss' energy in \eqref{weiss}. 
We recall the following results from \cite{gpps17,car24}.
 \begin{proposition}[Monotonicity of the Weiss' energy $\widetilde W_\mu$]\label{prop:monotonicity-weiss-tilde} Let $u$ be a solution to \eqref{def:sol-con-ostacolo}, with obstacle $\vf$ satisfying \eqref{e:hypo-phi}. Let $v=u^{(0)}$ given by \eqref{def:v} with $0\in\Lambda_\mu(u)$ and $\mu<k+\gamma$. Then \bea\label{monoW}\frac{d}{dr}\left(\widetilde W_\mu(v_r)+C_{\widetilde W}r^{k+\gamma-\mu}\right)\ge\frac{2}{r}\int_{\partial B_1} (\nabla v_r\cdot \nu -\mu v_r)^2\,d\mathcal{H}^n, \quad\text{for every} \quad r\in(0,1),\eea where $v_r$ is as in \eqref{def:rescalings2}, for some constant $C_{\widetilde W}=C_{\widetilde W}(v)>0$, depending only on $v$, $n$, $\mu$, $\varphi$, $k$ and $\gamma$. Moreover $C_{\widetilde W}(\widetilde v_\rho)\to0^+$ as $\rho\to0^+$, where $\widetilde v_\rho$ are the rescalings in \cref{prop:mono}.
\end{proposition}
 \begin{proposition}\label{prop:weiss-formula}  Let $u$ be a solution to \eqref{def:sol-con-ostacolo}, with obstacle $\vf$ satisfying \eqref{e:hypo-phi}. Let $v=u^{(0)}$ given by \eqref{def:v} with $0\in\Lambda_\mu(u)$ and $\mu<k+\gamma$. If $c_r:=v_r|_{\partial B_1}\in H^1(\partial B_1)$ is the trace of $v_r$, with $v_r$ as in \eqref{def:rescalings2}, then
   $$ \frac{d}{dr}\left(\widetilde W_\mu(v_r)+C_{\widetilde W}r^{k+\gamma-\mu}\right)\ge \frac{n+2\mu-1}{r}(W_\mu(z_r)-\widetilde W_\mu(v_r))+\frac{1}r\int_{\partial B_1}(\nabla v_r\cdot \nu-\mu v_r)^2\,d\HH^n,$$ for $r\in(0,1)$, where $z_r$ is the $\mu$-homogeneous extension of $c_r$ in $\R^{n+1}$.
\end{proposition}
\begin{proposition}\label{prop:useful} 
    Let $u$ be a solution to the thin obstacle problem \eqref{def:sol-con-ostacolo} with $\vf$ satisfying \eqref{e:hypo-phi}. Suppose that $0\in \Lambda_\mu(u)$, with $\mu< k+\gamma$. Let $v=u^{(0)}$ given by \eqref{def:v} and $v_r$ are the rescalings in \eqref{def:rescalings2}, then $$\int_{\partial B_1}|v_r-v_{r'}|\,d\HH^n\le C \log\left(\frac r{r'}\right)^{1/2}\left(\widetilde W_\mu(v_r)+C_{\widetilde W}r^{k+\gamma-\mu}\right)^{1/2}$$ for every $ 0<r'\le r\le 1$ and for some dimensional constant $C>0$.
\end{proposition}
\begin{proof} It is sufficient to integrate the identity from \cref{prop:monotonicity-weiss-tilde} and to apply the H\"older's inequality.
\end{proof}
 \section{Epiperimetric inequality for \texorpdfstring{$W_{2m+1}$}{Lg}}\label{section3}
 In this section we prove the epiperimetric inequality in \cref{thm:epi}.
 \subsection{Eigenfunctions and eigenvalues of $\Delta_{\mathbb S^{n}}$}\label{section}
   Let $p\in\mathcal{P}_{2m+1}$ and $T$ the operator as in \eqref{e:definition-of-T}. Then $T[p]:\mathbb{R}^{n}\to\R$ is a non-negative $2m$-homogeneous polynomial. We define 
   $$\mathcal{Z}_\delta:=\{T[p]\ge \delta\}\cap \partial B_1'\quad\text{and}\quad S_\delta:=\partial B_1\setminus \mathcal{Z}_\delta,$$ for every $\delta\ge0$. We also define the set \bea H^1_0(S_\delta):=\{\phi \in H^1(\partial B_1): \phi=0 \text{ on }\mathcal{Z}_\delta\}\subset H^1(\partial B_1),\eea for every $\delta\ge0$.
   If $\Delta_{\mathbb S^{n}}$ is the Laplace Beltrami operator on $\partial B_1$, 
   then there are a non-decreasing sequence $$0<\lambda_1^\delta\le\lambda_2^\delta\le\ldots\le\lambda_j^\delta\le\ldots$$ of eigenvalues (counted with multiplicity) and a sequence of eigenfunctions $\{\phi_j^{\delta}\}\subset H^1_0(S_\delta)$, which is an orthonormal basis in $H^1_0(S_\delta)$, such that \be\label{eq:eigenfunctions}\begin{cases}
	    -\Delta_{\mathbb S^{n}} \phi_j^{\delta}=\lambda_j^\delta \phi_j^{\delta}&\text{in }S_\delta,\\
        \phi_j^{\delta}=0&\text{in }\mathcal{Z}_\delta.
	\end{cases}\ee
We define the normalized eigenspace corresponding to the eigenvalue $\lambda$, as 
$$E_\delta(\lambda):=\{\phi^\delta\in H^1(\partial B_1): -\Delta_{\mathbb S^{n}} \phi^\delta=\lambda \phi^\delta, \ \phi^\delta=0 \text{ on }\mathcal{Z}_\delta ,\  \| \phi\|_{L^2({\partial B_1})}=1\},$$
for every $\delta\ge0$.
Notice that $H^1_0(S_\delta)$ is the natural Sobolev space where we can expand a trace $c\in H^1(\partial B_1)$ with eigenfunctions of $H^1_0(S_\delta)$.

	When $\delta=0$, i.e.~$\mathcal{Z}_0=\partial B_1'$,  we recover the spectrum on the half-sphere $\partial B_1^+$ (extended evenly with respect to $\{x_{n+1}=0\}$). We recall that if $\phi:\partial B_1\to\R $ is such that $\phi\equiv 0$ on $\partial B'_1$, then $r^{\alpha}\phi(\theta)$ is harmonic in $\mathbb{R}^{n+1}$ if and only if $\phi$ is an eigenfunction of the spherical Laplacian corresponding to the eigenvalue $\lambda(\alpha):=\alpha(n+\alpha-1)$. In this case $r^{\alpha}\phi(\theta)$ is a polynomial multiplied by $|x_{n+1}|$ and $\alpha\in \mathbb{N}$. This follows by extending $\phi$ to the whole ball as an odd function with respect to $\{x_{n+1}=0\}$ and using a Liouville-type theorem.
In particular, if $\{\phi_j\}\subset H^1_0(S_0)$ are the eigenfunctions on the half-sphere and $\lambda_j$ are the corresponding eigenvalues, then the following holds.
\begin{itemize}
		\item $\lambda_1=\lambda(1)$ and the corresponding eigenfunction is $\phi_1$ is a multiple of $|x_{n+1}|$.
		
		\item $\lambda_2=\ldots=\lambda_{n+1}=\lambda(2)$ and the corresponding eigenspace $E_0(\lambda(2))$ (of dimension $n$) coincides with the space generated by the restriction to $\partial B_1$ of two homogeneous harmonic polynomials multiples of $|x_{n+1}|$.
		
		\item In general, there exists an explicit function $f:\N\to\N$ such that 
     $$\lambda_{f(j-1)+1}=\ldots=\lambda_{f(j)}=\lambda(j),$$ 
  and the corresponding eigenspace $E_0(\lambda(j))$ is generated by the restriction to $\partial B_1$ of $j$-homogeneous harmonic polynomials multiples of $|x_{n+1}|$.
\end{itemize}
    In particular, we define \be\label{eq:d}\ell:=f(2m+1)\ee 
    as the number of eigenvalues with homogeneity less than or equal to $2m+1$. Thus, the eigenvalues $\lambda_1,\ldots,\lambda_\ell$ correspond to the homogeneities $1,\ldots,2m+1$, while $\lambda_{\ell+1},\ldots,\lambda_j,\ldots$ correspond to homogeneities greater than $2m+1$.

    In the following proposition, we prove that the eigenfunctions and the eigenvalues in $H^1_0(S_\delta)$ converge to the eigenfunctions and the eigenvalues on the half sphere. This is a consequence of the convergence of the resolvent operators.
     
     \begin{proposition}\label{lemma:convergence-of-s_delta} Let $\{\phi^\delta_j\}$ be the eigenfunctions in $H^1_0(S_\delta)$, according with \eqref{eq:eigenfunctions}. Let $\{\lambda^\delta_j\} $ be the eigenvalue corresponding to the eigenfunction $\{\phi_j^\delta\}$. Let $\{\lambda_j\}$ be the eigenvalues corresponding to $H^1_0(S_0)$, according with \eqref{eq:eigenfunctions}. Then, up to subsequences,
$$\phi_j^{\delta}\to\phi_j\quad\text{strongly in }L^2(\partial B_1)\qquad\text{and}\qquad \lambda_j^{\delta}\to\lambda_j \quad\text{for every}\quad j\in\N,$$as $\delta\to0^+$, where the sequence $\{\phi_j\}$ is an orthonormal basis of $H^1_0(S_0)$ of eigenfunctions corresponding to the eigenvalues $\{\lambda_j\}$.
     \end{proposition}
     \begin{proof}
         Consider a sequence $\delta_j\to 0^+$ and functions $f_j,f\in L^2(\partial B_1)$ such that $f_j$ converges to $f$ weakly in $L^2(\partial B_1)$. We define the functionals $F_j,F_\infty:L^2(\partial B_1)\to \R$ such that
         $$F_j(\psi):=\begin{cases}
             \int_{\partial B_1}|\nabla_{\theta} \psi|^2\,d\HH^n+\int_{\partial B_1}f_j\psi \,d\HH^n &\text{if }\psi\in H^1_0(S_{\delta_j}),\\
             +\infty &\text{otherwise},
         \end{cases}$$ and $$F_\infty(\psi):=\begin{cases}
             \int_{\partial B_1}|\nabla_{\theta}\psi|^2\,d\HH^n+\int_{\partial B_1}f\psi \,d\HH^n &\text{if }\psi\in H^1_0(S_{0}),\\
             +\infty &\text{otherwise},
         \end{cases}$$ and we prove that $F_j$ $\Gamma$-converges to $F$. 
        Indeed, the upper bound inequality follows by the inclusion $H^1_0(S_{0})\subset H^1_0(S_{\delta_j})$. For the lower bound inequality, we observe that
        if $\psi_j$ converges to $\psi$ in $L^2(\partial B_1)$ and $\|\psi_j\|_{H^1(\partial B_1)}\le C$, then $\psi_j$ converges to $\psi$ in $L^2(\partial B_1')$. In particular, if $\psi_j\in H^1_0(S_{\delta_j})$, then $\psi\in H^1_0(S_{0})$.
     
     The $\Gamma$-convergence of $F_j$ to $F$ implies the convergence of the minimizers. In our case this reads as follows. Let $f_j,f\in L^2(\partial B_1)$ be such that $f_j$ converges to $f$ weakly in $L^2(\partial B_1)$. Suppose that there is $\phi_j\in H_0^1(S_{\delta_j})$ such that
    \be\label{eq:eq1} \begin{cases}
         -\Delta_{\mathbb S^{n}}\phi_j=f_j &\text{in }S_{\delta_j},\\
         \phi_j=0&\text{in }\mathcal{Z}_{\delta_j}.
     \end{cases}\ee Then there is $\phi\in H^1(\partial B_1)$ such that $\phi_j$ converges to $\phi$ in $H^1(\partial B_1)$ and
     \be\label{eq:eq2} \begin{cases}
         -\Delta_{\mathbb S^{n}}\phi=f &\text{in }S_{0},\\
         \phi=0&\text{in }\mathcal{Z}_0.
     \end{cases}\ee Therefore, if $R_j,R:L^2(\partial B_1)\to L^2(\partial B_1)$ are the resolvent operators to the problems \eqref{eq:eq1} and \eqref{eq:eq2} respectively, then $$\|R_j(f_j)\to R(f)\|_{L^2(\partial B_1)}\quad\text{for every}\quad f_j\rightharpoonup f\quad\text{weakly in }L^2(\partial B_1),$$ as $j\to+\infty$. Then $$\|R_j(f_j)-R(f_j)\|_{L^2(\partial B_1)}\to0\quad\text{for every}\quad \|f_j\|_{L^2(\partial B_1)}\le 1$$ as $j\to+\infty$, i.e. $$\|R_j-R\|\to 0\quad\text{as}\quad j\to+\infty,$$ where $\|\cdot\|$ is the operator norm. 
     Once the convergence (in the operator norm) of the resolvent operators is proved, the claim follows by standard arguments.
     \end{proof}

 \subsection{Decomposition of $c$} Let $c\in H^1_0(S_\delta)$ be close to $p$ in $L^2(\partial B_1)$, i.e.~suppose that \eqref{eq:vicinanza} and \eqref{eq:quasizero} hold. Since the set of admissible blow-up $\mathcal{P}_{2m+1}$ is a subset of the set of the eigenfunctions of $H^1_0(S_0)$, we can take $\phi_\ell=p$, where $\ell$ is as in \eqref{eq:d}.
 To decompose the trace $c$ we use the following lemma.
\begin{lemma}\label{lemma:implicit} 
There is a sequence $\delta_k\to0^+$ such that the following holds. Suppose that $F:\R^\ell\to\R^\ell$ is such that 
$$F(\nu):=\left(\int_{\partial B_1}p_\nu \phi_1^\delta\,d\HH^n,\ldots,\int_{\partial B_1}p_\nu \phi_\ell^\delta\,d\HH^n\right),\quad\text{for some}\quad \delta\in\{\delta_k\},$$ 
where $\ell$ is as in \eqref{eq:d}, $p_\nu$ is defined as 
$$p_\nu(\theta):=\sum_{j=1}^{\ell}\nu_j\phi_j(\theta),$$ 
and $\phi_j^\delta$ and $\phi_j$ are the eigenfunctions of $-\Delta_{\mathbb S^{n}}$ for $H^1_0(S_\delta)$ and for the half sphere $H^1_0(S_0)$ respectively, according with \eqref{eq:eigenfunctions}. Then, there is a neighborhood of $e_\ell=(0,\ldots,0,1)\in\R^\ell$ such that $F$ is invertible there.
\end{lemma}
\begin{proof} Let $\{\delta_k\}$ the sequence for which \cref{lemma:convergence-of-s_delta} holds. 
    By the implicit function theorem, it is sufficient to prove that $DF$, the Jacobian matrix of $F$, is invertible. In particular, it is sufficient to show that for $\delta_k>0$ small enough, $DF\approx I$, where $I$ is the identity matrix in $\R^{\ell\times \ell}$. Using that $\partial_{\nu_j} p_\nu=\phi_j$ and applying \cref{lemma:convergence-of-s_delta}, we obtain that
    $$ \frac{\partial F_i}{\partial \nu_j}=\int_{\partial B_1}\phi_j\phi_i^\delta\,d\HH^n= \delta_{i,j}+o(1)$$ 
    as $\delta_k\to0^+$. Finally, the conclusion follows by extracting a subsequence for which $\delta_k$ is small enough.
\end{proof}
By \cref{lemma:implicit}, given $\delta\in\{\delta_k\}$, there is $\eps>0$ such that if $$\|c-\phi_\ell\|_{L^2(\partial B_1)}\le\eps,$$ then we can find constants $c_j\in\R$ such that $$\int_{\partial B_1}c(\theta)\phi_j^\delta\,d\HH^n=\int_{\partial B_1}p_\nu(\theta)\phi_j^\delta\,d\HH^n\quad\text{for every}\quad j=1,\ldots,\ell,$$ where $\nu=(c_1,\ldots,c_\ell)\in\R^\ell$.
Thus if we expand 
$$\phi(\theta):=c(\theta)-\sum_{j=1}^\ell c_j\phi_j(\theta)\in H^1_0(S_\delta)$$ 
using the orthonormal basis in $H^1_0(S_\delta)$, then $\phi$ contains only higher modes. 
In particular, since \eqref{eq:vicinanza} and \eqref{eq:quasizero} hold, we can decompose the trace $c$ as
\be\label{eq:decomposition}c(\theta)=P(\theta)+\phi(\theta),\quad \text{where}\quad P(\theta)=\sum_{j=1}^\ell c_j\phi_j\quad\text{and}\quad \phi(\theta)=\sum_{j=\ell+1}^\infty c_j\phi_j^\delta.\ee 
We choose $\delta\in\{\delta_k\}$ such that 
\be\label{eq:scelta-delta}\lambda_j^\delta\ge \lambda(2m+2)-1>\lambda(2m+3/2)\quad\text{for every}\quad j>\ell,
\ee and we choose the corresponding $\eps>0$ so that we can expand $c$ as above.

\subsection{Killing of lower and higher modes} We use the following two lemmas from \cite{csv20} to kill the lower and the higher modes respectively.	
\begin{lemma} \label{lemma:lower} Let $\{\phi_j\}\subset H^1_0(S_\delta)$ be the normalized eigenfunctions of $-\Delta_{\mathbb S^{n}}$ which are $0$ on $\mathcal{Z}_\delta$, according with \eqref{eq:eigenfunctions}, for some $\delta\ge0$. Let $\psi\in H^1_0(S_\delta)$ such that $$\psi(\theta)=\sum_{j=1}^\infty c_j \phi_j^\delta(\theta)$$ and let $r^{\mu}\psi(\theta)$ be the $\mu$-homogeneous extension of $\psi$ in $\R^{n+1}$. Then 
	\begin{equation*}
		W_\mu(r^\mu \psi)=\frac{1}{n+2\mu-1}\sum_{j=1}^\infty(\lambda_j^\delta-\lambda(\mu))c_j^2.
	\end{equation*}
\end{lemma}

\begin{lemma}\label{lemma:higher}
    Let $\{\phi_j^\delta\}\subset H^1_0(S_\delta)$ be the normalized eigenfunctions of $-\Delta_{\mathbb S^{n}}$ which are $0$ on $\mathcal{Z}_\delta$, according with \eqref{eq:eigenfunctions}, for some $\delta\ge0$. Let $\psi\in H^1_0(S_\delta)$ such that $$\psi(\theta)=\sum_{j=1}^\infty c_j \phi_j^\delta(\theta)$$ and let $r^{\mu}\psi(\theta)$ be the $\mu$-homogeneous extension of $\psi$ in $\R^{n+1}$. Then 
	\bea W_\mu(r^\alpha\psi)-(1-\kappa_{\alpha,\mu})W_\mu(r^\mu\psi)=\frac{\kappa_{\alpha,\mu}}{n+2\alpha-1}\sum_{j=1}^\infty(\lambda(\alpha)-\lambda_j^\delta)c_j^2,\eea where we set \be\label{eq:kappa}\kappa_{\alpha,\mu}:=\frac{\alpha-\mu}{n+\alpha+\mu-1}.\ee 
\end{lemma}
	\subsection{Killing of double product} Since the eigenfunctions $\phi_j$ and $\phi_j^\delta$ are not orthogonal in $H^1(\partial B_1)$ and in $L^2(\partial B_1)$, there is a bilinear form that appears in the decomposition of the Weiss' energy. In order to deal with this double product, in the proof of the epiperimetric inequality we will need the following lemma. 
 
 Given $v,w\in H^1(B_1)$ and $\mu>0$, we will use the following notation 
 \be\label{eq:rmu} R_\mu(v,w):=\int_{B_1}\nabla v\cdot \nabla w \,dx-\mu\int_{\partial B_1} vw\,d\HH^n.
 \ee
	\begin{lemma} \label{lemma:double-product} Let $\phi,\psi\in H^1(\partial B_1)$ be even with respect to $\{x_{n+1}=0\}$, with $$\phi(\theta)=\sum_{j=1}^\infty c_j\phi_j(\theta),$$
  where $\{\phi_j\}\subset H^1_0(S_0)$ are the normalized eigenfunctions of $-\Delta_{\mathbb S^{n}}$ which are $0$ on $\partial B_1'$, according to \eqref{eq:eigenfunctions}.
  Then $$R_\mu(r^\mu\phi(\theta),r^\alpha\psi(\theta))=\frac1{n+\alpha+\mu-1}\beta_\mu(\vf,\psi),$$ where $$ \beta_\mu(\phi,\psi):=\int_{\partial B_1}\sum_{j=1}^\infty(\lambda_j-\lambda(\mu))c_j\phi_j(\theta)\psi(\theta)\,d\HH^n-2\int_{\partial{B_1'}}(\partial_{\theta_{n+1}}\phi)\psi\,d\HH^{n-1}.$$
  	\end{lemma} 
		\begin{proof} By an integration by parts, we get
		\bea  R_\mu(r^\mu\phi,r^\alpha\psi)&=2\left(\int_{B_1^+}\nabla(r^\mu\phi)\cdot\nabla(r^\alpha\psi)\,dx-\mu\int_{(\partial B_1)^+}\phi\psi\,d\HH^n\right)
  \\&=-2\int_{B_1^+}\Delta(r^\mu\phi)r^\alpha\psi\,dx-2\int_{B_1'}\partial_{x_{n+1}}(r^\mu\phi)r^\alpha\psi\,d\HH^n\\&=
  -2\int_{B_1^+}\Delta\left(r^\mu\sum_{j=1}^\infty c_j\phi_j\right)r^\alpha\psi\,dx-2\int_{B_1'}\partial_{x_{n+1}}(r^\mu\phi)r^\alpha\psi \,d\HH^n.\eea
   Using the expression of the Laplacian and the gradient in spherical coordinates, we have that
  \bea R_\mu(r^\mu\phi,r^\alpha\psi)&=-2\int_{B_1^+}\left(\lambda(\mu)\sum_{j=1}^\infty c_j\phi_j+\Delta_{\mathbb S^{n}}\left(\sum_{j=1}^\infty c_j\phi_j\right)\right)r^{\mu-2}r^\alpha\psi\,dx\\
  &\qquad-2
  \int_{B_1'}\left(\mu\theta_{n+1}\phi+\partial_{\theta_{n+1}}\phi\right)\psi r^{\mu-1}r^\alpha\,d\HH^n\\&=-\frac1{n+\alpha+\mu-1}\int_{\partial B_1}\sum_{j=1}^\infty(\lambda(\mu)-\lambda_j)c_j\phi_j\psi\,d\HH^n
  \\&\qquad-\frac2{n+\alpha+\mu-1}\int_{\partial B_1'}(\partial_{\theta_{n+1}}
  \phi)\psi\,d\HH^{n-1},
		\eea where in the last equality
		we used that $\phi_j$ are eigenfunctions corresponding to the eigenvalues $\lambda_j$ and we integrated in $r$. We finally notice that the right-hand side in the last equality is precisely $\beta_\mu(\phi,\psi)$.
\end{proof}
Now we are ready to prove \cref{thm:epi}.
\begin{proof}[Proof of \cref{thm:epi}] Let $c\in H^1(\partial B_1)$ and $z$ its $(2m+1)$-homogeneous extension. Since \eqref{eq:vicinanza} and \eqref{eq:quasizero} hold, we can decompose $c$ as in \eqref{eq:decomposition}. Then $$z(r,\theta)=r^{2m+1}P(\theta)+r^{2m+1}\phi(\theta)$$ and we define the competitor \bea \zeta(r,\theta):=r^{2m+1}P(\theta)+r^{\alpha}\phi(\theta),\eea where $\alpha:=2m+3/2$. Notice that $\zeta\ge0$ on $B_1'$ since $P(\theta)\equiv 0$ on $B_1'$.
So we only need prove the epiperimetric inequality in \eqref{eq:epi}.
We also set $\mu:=2m+1$ and $\kappa_{\alpha,\mu}$ as in \eqref{eq:kappa}. Then, the energy can be decomposed as 
$$W_{\mu}(r^{\mu}P+r^\alpha\phi)=W_{\mu}(r^{\mu}P)+W_{\mu}(r^\alpha\phi)+2R_{\mu}(r^{\mu}P,r^\alpha\phi),$$ 
where $R_{\mu}$ is defined in \eqref{eq:rmu}. 
Therefore \be\label{eq:finale}\begin{aligned}W_\mu(\zeta)-(1-\kappa_{\alpha,\mu})W_\mu(z)&=W_{\mu}(r^{\mu}P+r^\alpha\phi)-(1-\kappa_{\alpha,\mu})W_{\mu}(r^{\mu}P+r^\mu\phi)\\&=\kappa_{\alpha,\mu}W_{\mu}(r^{\mu}P)+W_{\mu}(r^{\alpha}\phi)-(1-\kappa_{\alpha,\mu})W_{\mu}(r^{\mu}\phi)\\ &\qquad+2R_\mu(r^{\mu} P,r^\alpha\phi)-2(1-\kappa_{\alpha,\mu})R_\mu(r^{\mu} P,r^\alpha\phi).
\end{aligned}
\ee
First, by \cref{lemma:lower}, we observe that \be\label{eq:eq10}W_\mu(r^\mu P)\le 0.\ee
Moreover, by \cref{lemma:higher}, we have that  \be\label{eq:eq11}W_{\mu}(r^{\alpha}\phi)-(1-\kappa_{\alpha,\mu})W_{\mu}(r^{\mu}\phi)=\frac{\kappa_{\alpha,\mu}}{n+2\alpha-1}\sum_{j=1}^\infty(\lambda(\alpha)-\lambda_j^\delta)c_j^2\le 0,\ee by \eqref{eq:scelta-delta}.
Finally, notice that by \cref{lemma:double-product} and by definition of $\kappa_{\alpha,\mu}$, we have that
\be\label{eq:eq12}\begin{aligned}  R_\mu(r^{\mu} P,r^\alpha\phi)-(1-&\kappa_{\alpha,\mu})R_\mu(r^{\mu} P,r^\alpha\phi)\\
&=-\left(\frac 1{n+\alpha+\mu-1}-(1-\kappa_{\alpha,\mu}) \frac 1{n+2\mu-1}\right)\beta_\mu(P,\phi)=0,
\end{aligned}\ee
which concludes the proof by using \eqref{eq:finale} with \eqref{eq:eq10}, \eqref{eq:eq11} and \eqref{eq:eq12}.
\end{proof}
\section{Application to the epiperimetric inequality}\label{section4}
In this section we show that we can apply the epiperimetric inequality in \cref{thm:epi} at every trace $ (\widetilde v_\rho)_r|_{\partial B_1},$ with 
\be\label{eq:doublerescalings} (\widetilde v_\rho)_r:= \frac{\widetilde v_\rho(rx)}{r^{2m+1}}
\ee
where $\widetilde v_\rho$ is as in \cref{prop:mono}.
In particular, we prove the following proposition.
\begin{proposition}\label{prop:every-rescaled-fundamental} Let $u$ be a solution to the thin obstacle problem \eqref{def:sol-con-ostacolo}, with obstacle $\vf$ satisfying \eqref{e:hypo-phi}. Suppose that $0\in\Lambda_{2m+1}(u)$, $2m+1\le k$ and $v=u^{(0)}$ given by \eqref{def:v}. Then there is $\rho>0$ small enough such that the epiperimetric inequality in \cref{thm:epi} can be applied to the sequence of the traces $  (\widetilde v_\rho)_r|_{\partial B_1}$, defined in \eqref{eq:doublerescalings}, for every $r\in(0,1).$
    
\end{proposition}
To prove \cref{prop:every-rescaled-fundamental}, we use the following fundamental proposition.
\begin{proposition}\label{prop:every-rescaled} For every $H_0>0$ and $\phi_0>0$ there are constants $\eta_1>0$, $\eta_2>0$, $\delta_1>0$ and $\rho_0>0$, depending only on $H_0$, $\phi_0$, $n$, $m$, $\varphi$, $k$ and $\gamma$, such that the following holds.
Let $u$ be a solution to the thin obstacle problem \eqref{def:sol-con-ostacolo}, with obstacle $\vf$ satisfying \eqref{e:hypo-phi}. Suppose that $0\in\Lambda_{2m+1}(u)$, $2m+1\le k$ and $v=u^{(0)}$ given by \eqref{def:v} with $\widetilde v_\rho$ as in \cref{prop:mono}. 
We also suppose that, for some $p\in\mathcal{P}_{2m+1}$, with $\|p\|_{L^2(\partial B_1)}=1$, we have $$\|\widetilde v_{\rho}-p\|_{L^2(\partial B_1)}\le \eta_1, \quad \|\widetilde v_{\rho}-p\|_{L^2(B_2)}\le \eta_2,\quad\text{for some}\quad \rho\in(0, \rho_0),$$ and \bea H(2,(\widetilde v_{\rho})_r)\le H_0,\quad \phi(2r, \widetilde v_\rho)\le \phi_0, \quad \widetilde W_{2m+1}(\widetilde v_{ \rho})+C_{\widetilde W}(\widetilde v_\rho)\le\delta_1\quad\text{for every}\quad r \in(0,1)\eea
where $C_{\widetilde W}(\widetilde v_\rho)>0$ is as in \cref{prop:monotonicity-weiss-tilde}.
    Then the epiperimetric inequality in \cref{thm:epi} can be applied to the sequence of the traces $ (\widetilde v_{\rho})_r|_{\partial B_1}$, defined in \eqref{eq:doublerescalings}, for every $r\in(0,1).$
\end{proposition}
We need some preliminary lemmas. First, we show that the norms $\|(\widetilde v_\rho)_r-p\|_{L^{\infty}}$ are controlled by $\|(\widetilde v_\rho)_r-p\|_{L^2}$. We notice that we only need a modulus of continuity, which we obtain via a simple argument in the next lemma.  
\begin{lemma}\label{lemma:linftyl2} There is a dimensional constant $\sigma\in(0,1)$ such that the following holds.
    Let $u$ be a solution to the thin obstacle problem \eqref{def:sol-con-ostacolo} with obstacle $\vf$ satisfying \eqref{e:hypo-phi}. Let $v=u^{(0)}$ given by \eqref{def:v} with $(\widetilde v_\rho)_r$ as in \eqref{eq:doublerescalings}. 
    We also suppose that
    $$H(2, (\widetilde v_\rho)_r)\le H_0\quad\text{and}\quad \phi(2r,\widetilde v_\rho)\le \phi_0\quad\text{for some}\quad\rho\in\left(0,\frac12\right),\quad\text{for every}\quad r\in(0,1).$$ 
   If $p\in\mathcal{P}_{2m+1},$ with $\|p\|_{L^2(\partial B_1)}=1$, then 
   $$\|(\widetilde v_\rho)_r-p\|_{L^{\infty}(B_{3/2}\setminus B_{1/4})}\le C\|(\widetilde v_\rho)_r-p\|_{L^2(B_2\setminus B_{1/8})}^\sigma\quad\text{for every}\quad r\in(0,1),$$ 
   for a constant $C>0$ depending only on $H_0$, $\phi_0$, $n$, $m$, $\varphi$, $k$ and $\gamma$.
\end{lemma}
\begin{proof}
    Notice that, in general, if $G:\R^{n+1}\to\R$ is a non-negative $L$-Lipschitz continuous function, $x_0\in \R^{n+1}$ and $M:=G(x_0)$, then 
    $$\int_{B_R(x_0)}G^2(x)\, dx\ge C\frac{{M}^{n+3}}{L^{n+1}},$$ 
    where $R=M/L$ (see e.g. \cite[Lemma 3.2]{sv21}). Thus, if for instance
    $$M:=\|(\widetilde v_\rho)_r-p\|_{L^\infty(B_{3/2}\setminus B_{1/4})}=(\widetilde v_\rho)_r(x_0)-p(x_0)\quad\text{for some}\quad x_0\in \overline B_{3/2}\setminus B_{1/4},$$ 
    we can choose $G:=((\widetilde v_\rho)_r-p)_+.$ If $L$ is the Lipschitz constant of $((\widetilde v_\rho)_r-p)_+$ in $B_{3/2}$, by \cref{prop:blow-up-homo2}, $M$ and $L$ are bounded by a constant that depends only on $H_0$, $\phi_0$, $n$, $m$, $\vf$, $k$ and $\gamma$. Then, up to enlarge $L>0$, we can take $R=M/L>0$ small enough. Finally, the claim follows from the previous estimate, with $\sigma=\frac1{n+3}$.
\end{proof}
In the next lemma we show that if the $L^\infty$ distance between $\widetilde v_\rho$ and $p\in \mathcal P_{2m+1}$ is small, then the positivity set of $\widetilde v_\rho$ is contained in a neighborhood of $\{T[p]=0\}$.
\begin{lemma}\label{lemma:B.3} There are constants $\eta_3>0$ and $\overline \rho>0$, depending only on $H_0$, $\phi_0$, $n$, $m$, $\varphi$, $k$ and $\gamma$, such that the following holds. 
Let $u$ be a solution to the thin obstacle problem \eqref{def:sol-con-ostacolo}, with obstacle $\vf$ satisfying \eqref{e:hypo-phi}. Suppose that $0\in\Lambda_{2m+1}(u)$, $2m+1\le k$ and $v=u^{(0)}$ given by \eqref{def:v} with $(\widetilde v_\rho)_r$ as in \eqref{eq:doublerescalings}. We also suppose that $$\|(\widetilde v_\rho)_r-p\|_{L^\infty(B_{3/2}\setminus B_{1/4})}\le \eta_3\quad\text{for some}\quad \rho\in(0,\overline \rho), \ r\in(0,1).$$ Then 
$$ (\widetilde v_\rho)_{r'}\equiv 0\quad\text{in}\quad \mathcal{Z}_\delta:=\{T[p]\ge\delta\}\cap\partial B_1'\quad\text{for every}\quad r'\in\left( \frac13r, r\right),$$
where $\delta>0$ is as in \cref{thm:epi} and $T$ is the operator from \eqref{e:definition-of-T}.
\end{lemma}
\begin{proof}
The proof is similar to \cite[Lemma B.3]{frs20}. 
    Let $z=(z',0)\in  B_1'\setminus B_{1/3}'$ be such that 
    \bea T[p](z')\ge \frac{\delta}{3^{2m}}.\eea
    Consider the function 
    $$\phi_C(x):=-(n+1)|x_{n+1}|^2+|x'|^2+C,$$ 
    for every $C>0$. Then, we have that
    $$\widetilde v_\rho(r x+r z)\le \phi_C(x)\quad\text{for every}\quad x\in\partial B_{r_1},$$ 
    for some $r_1>0$ and $\eta_3>0$ small enough, by the hypothesis assumption. Next, suppose that there is $C_\ast>0$ such that the function $\phi_{C_\ast}$ touches $\widetilde v_\rho(r \cdot+r z)$ from above. 
Notice that the contact point $x_0$ cannot lie in $B_{r_1}\setminus\{x': \ \widetilde v_\rho(r x'+r z',0)=0\}$, since the right hand side of $\widetilde v_\rho(r \cdot+rz)$ is small (for $\overline \rho$ small enough), while  $ \Delta\phi_{C_\ast}=-2$. On the other hand, if $\phi_{C_\ast}$ touches $\widetilde v_\rho(r x'+r z',0)$ in $x_0\in \{x':\ \widetilde  v_\rho(r x'+r z',0)=0\}$, then $\phi_{C_\ast}>0$, which is a contradiction. Thus, $\phi_{C_\ast}$ 
cannot touch $\widetilde v_\rho(r \cdot+r z)$ from above when $C_\ast>0$ and so, we get
        $$\widetilde v_\rho(r x+r z)\le \phi_0(x)\quad\text{for every}\quad x\in B_{r_1}.$$ 
   Since $\phi_0(0)=0$, this implies that $\widetilde v_\rho(r z)=0$.

    Now, given $x\in \mathcal{Z}_\delta$ and $r'\in(\frac13r,r)$, we take $z=(z',0)=\frac {r'}{r} x\in B_1'\setminus B_{1/3}'$, then $$T[p](z')=\left(\frac {r'}{r}\right)^{2m}T[p](x')\ge \frac\delta{3^{2m}}.$$ 
    Therefore $\widetilde v_\rho(r'x)=\widetilde v_\rho(r z)=0$, which concludes the proof.
    \end{proof}
In the next lemma we show that if $\widetilde v_\rho$ is close to $p$ at some scale, then it stay close to $p$ at some smaller scale.
\begin{lemma}\label{lemma:tec} For every $\beta>0$ there are constants $\delta_1>0$ and $\delta_2>0$, depending only on  $H_0$, $\phi_0$, $n$, $m$, $\varphi$, $k$ and $\gamma$, such that the following holds.
Let $u$ be a solution to the thin obstacle problem \eqref{def:sol-con-ostacolo}, with obstacle $\vf$ satisfying \eqref{e:hypo-phi}. Suppose that $0\in\Lambda_{2m+1}(u)$ with $=2m+1\le k$ and $v=u^{(0)}$ given by \eqref{def:v} with $(\widetilde v_\rho)_r$ as in \eqref{eq:doublerescalings}. We also suppose that $$\widetilde W_{2m+1}(\widetilde v_{\rho})+C_{\widetilde W}(\widetilde v_\rho)\le \delta_1,\quad\|(\widetilde v_{\rho})_r-p\|_{L^2(\partial B_1)}\le\delta_2\quad\text{for some}\quad\rho\in\left(0,\frac12\right), \ r\in(0,1),$$ 
where $C_{\widetilde W}(\widetilde v_\rho)>0$ is as in \cref{prop:monotonicity-weiss-tilde}. Then
$$\|(\widetilde v_{\rho})_{r'}-p\|_{L^2(\partial B_1)}\le \beta\quad\text{for every}\quad r'\in\left(\frac18r,r\right).$$ 
\end{lemma}
\begin{proof} First notice that by \cref{prop:useful}, we have that
    \begin{align*}\|(\widetilde v_{\rho})_r-(\widetilde v_{\rho})_{r'}\|_{L^2(\partial B_1)}&\le C\log\left(\frac{r}{r'}\right)^{1/2}\left(\widetilde W_{2m+1}((\widetilde v_{\rho})_r)+C_{\widetilde W}(\widetilde v_\rho)r^{k+\gamma-2m-1}\right)^{1/2}
    \\&\le C\log\left(8\right)^{1/2}\left(\widetilde W_{2m+1}(\widetilde v_{\rho})+C_{\widetilde W}(\widetilde v_\rho)\right)^{1/2}
    \\&\le C\log\left(8\right)^{1/2}\delta_1^{1/2}
    \quad\text{for every}\quad r'\in \left(\frac18r,r\right), \end{align*} where in the second last inequality we used \cref{prop:monotonicity-weiss-tilde}.
Therefore \begin{align*}\|(\widetilde v_{\rho})_{r'}-p\|_{L^2(\partial B_1)}&\le \|(\widetilde v_{\rho})_r-(\widetilde v_{\rho})_{r'}\|_{L^2(\partial B_1)}+\|(\widetilde v_{\rho})_r-p\|_{L^2(\partial B_1)}\\ &\le C \log\left(8\right)^{1/2}\delta_1^{1/2}+\delta_2\le \beta \end{align*} if we choose $\delta_1>0$, $\delta_2>0$ and $\overline \rho>0$ small enough.
\end{proof}

Now we are ready to prove \cref{prop:every-rescaled}.
\begin{proof}[Proof of \cref{prop:every-rescaled}]
Let $\rho\in(0, \rho_0)$ as in the hypothesis, with $\rho_0>0$ to be chosen and such that we can apply \cref{lemma:B.3}.
First notice that by \cref{lemma:linftyl2} and \cref{lemma:B.3}, we can find $\eta_2>0$ such that if $$\|(\widetilde v_{\rho})_r-p\|_{L^2(B_2\setminus B_{1/8})}\le \eta_2\quad\text{for some}\quad r\in(0,1),$$ then $$ (\widetilde v_{\rho})_{r'}\equiv 0\quad\text{in } \mathcal{Z}_\delta\quad\text{for every}\quad r'\in\left(\frac 13 r, r\right).$$
  Let $\beta\in(0,\eps)$ to be chosen and take the corresponding $\delta_1,$ $\delta_2$ as in \cref{lemma:tec}. We set $\eta_1\in(0,\delta_2)$ to be chosen.
  By \cref{lemma:tec}, we know that if we have the bounds
  \be\label{eq:applic}\|(\widetilde v_{\rho})_{r}-p\|_{L^2( \partial B_1)}\le \delta_2\quad\text{and}\quad\|(\widetilde v_{\rho})_{r}-p\|_{L^2( B_2\setminus B_{1/8})}\le \eta_2\quad\text{for some}\quad r\in(0, 1),\ee 
  then 
  we can apply \cref{thm:epi} to all the traces $ (\widetilde v_{\rho})_{r'}|_{\partial B_1}$ with $r'\in(\frac13r,r)$.
  
  We define $r_0\in[0,1]$ as the smallest number such that we can apply the epiperimetric inequality in \cref{thm:epi} to the traces $(\widetilde v_{\rho})_r|_{\partial B_1}$ for $r\in (r_0,1]$. Since \eqref{eq:applic} is satisfied for $r=1$, we can apply the epiperimetric inequality for $r\in (\frac13,1]$, so we have that $r_0\le \frac13$. 
  We will show that $r_0=0$. 
  
  Suppose by contradiction that $r_0>0$. Using the Weiss' formula in \cref{prop:weiss-formula} together with the epiperimetric inequality in \cref{thm:epi} and integrating in $r$ (see e.g. \cite{car24}), we obtain
  $$\widetilde W_{2m+1}((\widetilde v_{\rho})_r)\le C(\rho)r^{\alpha},\quad\text{for every}\quad r\in\left(r_0,1\right),$$ 
  for some constants $C(\rho)>0$ and $\alpha>0$, with $C(\rho)\to0^+$ as $\rho\to0^+$ (we used  $\widetilde W_{2m+1}(\widetilde v_\rho)\to0$ and $C_{\widetilde W}(\widetilde v_\rho)\to0^+$ as $\rho\to0^+$). By  \cref{prop:useful} and a dyadic argument, we obtain that 
  $$\int_{\partial B_1} | \widetilde v_{\rho}-(\widetilde v_{\rho})_r| \,d\HH^n\le C(\rho)\quad\text{for every}\quad r\in\left(r_0,1\right),$$ 
  where $C(\rho)\to 0^+$ as $\rho\to0^+$.
 Therefore, for every $ r\in(r_0,1)$ 
 \begin{align*}
  \|(\widetilde v_{\rho})_r-p\|_{L^2(\partial B_1)}&\le\|\widetilde v_{\rho}-(\widetilde v_{\rho})_r\|_{L^2(\partial B_1)}+\|\widetilde v_{\rho}-p\|_{L^2(\partial B_1)}\\
  &\le C(\rho)+ \eta_1
  \le \frac{\delta_2}2+ \eta_1\le\delta_2,
  \end{align*} 
where we chose $\eta_1\le \frac{\delta_2}2$ and $\rho_0>0$ small enough such that $C(\rho)\le \frac{\delta_2}2$ for all $\rho\le\rho_0$.
  Then, by \cref{lemma:tec}, we have that
  \be\begin{aligned}\label{eq:finale1} \|(\widetilde v_{\rho})_r-p\|_{L^2(\partial B_1)}\le \beta\le\eps\quad\text{for every}\quad  r\in\left(\frac{1}{8}r_0,1\right).
  \end{aligned}
  \ee 
Integrating in polar coordinates and applying \eqref{eq:finale1} to all $r\in(r_0,1/2)$, we get 
\bea\label{eq:finale2}\begin{aligned} \|(\widetilde v_{\rho})_r-p\|_{L^2(B_2\setminus B_{1/8})}&=\left(\int_{1/8}^2\|(\widetilde v_{\rho})_r-p\|^2_{L^2(\partial B_t)}\,dt\right)^\frac12\\
&=\left(\int_{1/8}^2t^{n+2m+1}\|(\widetilde v_{\rho})_{rt}-p\|^2_{L^2(\partial B_1)}\,dt\right)^\frac12\\
&\le \left(\int_{1/8}^2 t^{n+2m+1}\beta^2\,dt\right)^\frac12\\
&=C\beta\le\eta_2\quad\text{for every}\quad r\in\left(r_0,\frac{1}2\right),
\end{aligned}\eea 
for $\beta>0$ small enough.
 Thus, \eqref{eq:applic} is satisfied for every $\rho\in(r_0,\frac12)$ and so we can apply the epiperimetric inequality from \cref{thm:epi} in the interval $(\frac{1}3r_0,1]$, which is a contradiction with the definition of $r_0$.
  \end{proof}
    Finally we can use \cref{prop:every-rescaled} to prove \cref{prop:every-rescaled-fundamental}. 
    \begin{proof}[Proof of \cref{prop:every-rescaled-fundamental}] 
Let $H_0:=2^{n+2(2m+1)+1}$ and $\phi_0:=n+2(2m+1)+1$, we take the corresponding $\eta_1>0$, $\eta_2>0$, $\delta_1>0$ and $\rho_0>0$ as in \cref{prop:every-rescaled}. By \cref{prop:mono}, for every $\rho\in(0,\rho_1)$ and for every $r\in(0,1)$, with $\rho_1<\rho_0$ small enough to be chosen, we have $$H(2,(\widetilde v_\rho)_r)= \frac{H(2r,\widetilde v_\rho)}{r^{n+2(2m+1)}}=\frac1{r^{n+2(2m+1)}}\frac{H(2r\rho,v)}{H(\rho,v)}\le H_0\quad\text{and}\quad \phi(2r,\widetilde v_\rho)= \phi(2\rho r,v)\le \phi_0.$$ Moreover, by \cref{prop:mono} and \cref{prop:monotonicity-weiss-tilde} we get \begin{align*}\widetilde W_{2m+1}(\widetilde v_{\rho})+C_{\widetilde W}(\widetilde v_\rho)&=
 \left(\rho\frac{\mathcal{I}( \rho, v)}{H(\rho, v)}-(2m+1)\right)+C_{\widetilde W}(\widetilde v_\rho)
 \\&\le \left( \frac12(\phi(\rho_1,v)-n)-(2m+1)\right)+\frac{\delta_1}2
\\&\le \delta_1,\end{align*} for every $\rho\in(0,\rho_1)$, if $\rho_1>0$ is small enough.
 Moreover we also have $$\|\widetilde v_{\rho}-p\|_{L^2(\partial B_1)}\le \eta_1\quad\text{and}\quad\|\widetilde v_{\rho}-p\|_{L^2(B_2)}\le \eta_2\quad\text{ for some}\quad \rho\in(0, \rho_1)$$ for some $p\in\mathcal{P}_{2m+1}$, with $\|p\|_{L^2(\partial B_1)}=1$, since $\widetilde v_\rho$ converge, up to subsequences, to some $(2m+1)$-homogeneous global solution (see \cref{prop:mono}). Then the hypotheses of \cref{prop:every-rescaled} are satisfied and we conclude. 
    \end{proof}

\section{Rate of convergence and stratification}\label{section5}
In this section we prove that the epiperimetric inequality in \cref{thm:epi} implies the rate of convergence in \cref{prop:rate} and the stratification of the contact set in \cref{cor:stratification}. 
 Once we know that we can apply the epiperimetric inequality in \cref{thm:epi}, the proofs are standard (see e.g. \cite{gps16,fs16,gpps17,csv20,car24}). We briefly sketch the proofs here.
 
  \begin{proof}[Proof of \cref{prop:rate}]
 By \cref{prop:every-rescaled-fundamental},
 as in the proof of \cref{prop:every-rescaled}, if $0\in\Lambda_{2m+1}(u)$ with $2m+1\le k$, we deduce that $$\widetilde W_{2m+1}((\widetilde v_\rho)_{r})\le Cr^\alpha\quad\text{for every}\quad r\in(0,1),$$  for some $\rho>0$, where $(\widetilde v_\rho)_{r}$ is as in \eqref{eq:doublerescalings}.
 Since $$\widetilde W_{2m+1}(v_{r\rho})=\frac{H(\rho,v)}{\rho^{n+2(2m+1)}}\widetilde W_{2m+1}((\widetilde v_\rho)_r),$$
 then the same decay can be deduced for the sequence $v_r$ for every $r\in(0,\rho)$.
 Reasoning as in the proof of \cref{prop:every-rescaled}, we get 
 $$ \int_{\partial B_1} | v_{r}-p| \,d\HH^n\le Cr^{\alpha}
 \quad\text{for every}\quad  r\in(0,\rho),$$ 
 where $p$ is the blow-up limit of $v$. 
 As a consequence we obtain the rate of convergence in $L^2(\partial B_1)$ and in $L^\infty(B_1)$, as in the proof of \cref{lemma:linftyl2}.
 \end{proof}
\begin{proof}[Proof of \cref{cor:stratification}]
As in the proof of \cref{prop:rate}, we have that if $K\subset\Lambda_{2m+1}(u)\cap \R^n$ is a compact set and $2m+1\le k$, then
$$ \int_{\partial B_1} | v_{x_0,r}-p_{x_0}| \,d\HH^n\le Cr^{\alpha }
 \quad\text{for every}\quad x_0\in\Lambda_{2m+1}(u)\cap K,\ r\in(0,\rho),$$ 
 where $\rho>0$, $v_{x_0,r}=\frac{v(x_0+rx)}{r^{2m+1}}$ and $p_{x_0}$ is the blow-up limit of $v$ at $x_0$. 
  The stratification of the set $\Lambda_{2m+1}(u)$ now follows from the implicit function theorem and the Whitney extension theorem (see for instance \cite{gp09,csv20}).
  \end{proof}

\section{Frequency gap}\label{section6} 
This section is dedicated to the proof of \cref{prop:gap}. Key points of the proof are \cref{thm:epi} and the following epiperimetric inequality for negative energies.

\begin{proposition}[Epiperimetric inequality for negative energies $W_{2m+1}$]\label{prop:epi-negative} There are constants $ \eps>0$, $\delta>0$, $\kappa>0$ and $\eta>0$, depending only on $n$ and $m$, such that the following holds.
     Let $c\in H^1(\partial B_1)$, with $c\ge 0$ on $B_1'$ and $c$ even with respect to $\{x_{n+1}=0\}$. Let $z(r,\theta)=r^{2m+1}c(\theta)$ be the $(2m+1)$-homogeneous extension in $\R^{n+1}$ of $c$. We suppose that 
     \be\label{eq:vicinanza-neg}\|c-p\|_{L^2(\partial B_1)}\le \eps\quad\text{for some}\quad p\in\mathcal{P}_{2m+1},\ee
     and 
     \be\label{eq:quasizero-neg} c\equiv 0 \quad\text{on}\quad \mathcal{Z}_\delta:=\{T[p]\ge\delta\}\cap \partial B_1', \ee
     with $\|p\|_{L^2(\partial B_1)}=1$ and $T$ is the operator in \eqref{e:definition-of-T}.
     If \be\label{eq:lowerbound}|W_{2m+1}(z)|\le \eta,\ee then there is a function $\zeta\in H^1(B_1)$ such that 
     \bea\label{eq:epi-neg}
     W_{2m+1}(\zeta)\le (1+|W_{2m+1}(z)|)W_{2m+1}(z),
     \eea 
     where $\zeta\ge 0$ on $B_1'$, $\zeta=c$ on 
     $\partial B_1$ and $\zeta$ is even with respect to $\{x_{n+1}=0\}$.
\end{proposition}
\begin{proof}
    The proof is similar to the one in \cref{thm:epi}. 
    We first observe that we can suppose $W(z)< 0$, since otherwise one can simply choose $\zeta=z$.
    As in \cref{lemma:implicit}, using \eqref{eq:vicinanza-neg}, \eqref{eq:quasizero-neg} and \cref{lemma:convergence-of-s_delta}, we can decompose $c$ as $$c(\theta)=h(\theta)+\phi(\theta),$$ where $$h(\theta)=c_\ell \phi_\ell\quad\text{and}\quad
    \phi(\theta)=\sum_{j\not=\ell} c_j\phi_j^\delta, $$ where $\ell$ is defined as in \eqref{eq:d}. 
   We set $\mu:=2m+1$ and we define the competitor $$\zeta(r,\theta)=r^{\mu}h(\theta)+r^{\alpha}\phi(\theta),$$ where $\alpha$ is such that
    $$|W_{\mu}(z)|= \kappa_{\mu,\alpha},$$
    where $\kappa_{\mu,\alpha}$ is given by \eqref{eq:kappa}. Now, since
    $$\frac{\mu-\alpha}{\alpha+\mu+n-1}=\kappa_{\mu,\alpha}\le\eta,$$
    by choosing $\eta$ small enough we get $\alpha\in(2m,2m+1)$.

We notice that $\zeta$ is an admissible competitor since $$\zeta=r^\alpha \phi=r^\alpha c\ge0\quad\text{on } B_1'.$$

   Defining the operator $R$ as in \eqref{eq:rmu} and using that $W_{\mu}(r^{\mu}h)=0$, we get
    \begin{align*}R_\mu(r^\mu h,r^\mu \phi)&=R_\mu(r^\mu h,r^\mu c)= -\int_{B_1}\Delta (r^{\mu}h)r^{\mu}c\,dx=-2\int_{B_1'}\partial_{x_{n+1}}(r^\mu h)r^\mu c\,d\HH^n\ge0,
    \end{align*} since $r^\mu h$ is a solution to \eqref{def:sol} and $c\ge0$ on $B_1'$. 
    Then \be\label{eq:101}0>W_\mu(z)=W_{\mu}(r^\mu\phi)+R_{\mu}(r^\mu h,r^\mu\phi)\ge W_{\mu}(r^\mu\phi).\ee
    Using again that $r^\mu h$ has zero Weiss' energy, we obtain
    $$W_{\mu}(r^\mu h+r^\alpha\phi)=W_{\mu}(r^{\alpha}\phi)+ 2R_{\mu}(r^{\mu}h,r^{\alpha}\phi).$$ 
    Then, by \cref{lemma:higher} and \cref{lemma:double-product}, there is a constant $C>0$, depending only on $n$ and $m$, such that \begin{align*} W_{\mu}(\zeta)-(1+\kappa_{\mu,\alpha})W_{\mu}(z)&= W_{\mu}(r^{\alpha}\phi)-(1+\kappa_{\mu,\alpha})W_{\mu}(r^{\mu}\phi)\\&=\frac{-\kappa_{\mu,\alpha}}{n+2\alpha-1}\sum_{j=1}^\infty(\lambda(\alpha)-\lambda_j^\delta)c_j^2\\&=\frac{\kappa_{\mu,\alpha}}{n+2\alpha-1}\left(\sum_{j=1}^\infty(\lambda_j^\delta-\lambda(\mu))c_j^2+\sum_{j=1}^\infty(\lambda(\mu)-\lambda(\alpha))c_j^2\right)\\&=\frac{n+2\mu-1}{n+2\alpha-1} \kappa_{\mu,\alpha}W_{\mu}(r^\mu\phi)+C\kappa_{\mu,\alpha}^2\|\phi\|_{L^2(\partial B_1)}^2,
    \end{align*} where in the last equality we used \cref{lemma:lower}.
    Combining the above estimate with \eqref{eq:vicinanza-neg} and \eqref{eq:101}, we get that
    \bea 
    W_{\mu}(\zeta)-(1+\kappa_{\mu,\alpha})W_{\mu}(z)&\le \kappa_{\mu,\alpha}W_{\mu}(r^\mu\phi)+C\kappa_{\mu,\alpha}^2\eps\\
    &\le \kappa_{\mu,\alpha}W_{\mu}(z)+C\kappa_{\mu,\alpha}^2\eps\\
    &=-|W_{\mu}(z)|^2+C|W_{\mu}(z)|^2\eps\\&= |W_{\mu}(z)|^2(-1+C\eps)\\&\le0
    \eea   
    since $\eps>0$ is small enough.
\end{proof}
To show the frequency gap, we will use the following lemma from \cite{csv20} with the epiperimetric inequalities in \cref{thm:epi} and \cref{prop:epi-negative}.
\begin{lemma}\label{lemma:gap}
    	Let $c\in H^1(\partial B_1)$ such that $r^{\mu+t}c$ is a solution to the thin obstacle problem \eqref{def:sol}, then 
		\begin{equation*}\label{terza}
			W_\mu(r^{\mu+t}c)=t\| c\|_{L^2(\partial B_1)}^2
		\quad\text{and}\quad
			W_\mu(r^{\mu}c)=\left(1+\frac{t}{n+2\mu-1}\right)W_\mu(r^{\mu+t}c).
		\end{equation*}
	\end{lemma}
\begin{proof}[Proof of \cref{prop:gap}] By contradiction, suppose that there are functions $u_k$ and a sequence $t_k\to0$, such that $u_k$ is global $(2m+1+t_k)$-homogeneous solution to the thin obstacle problem \eqref{def:sol}. Without loss of generality we can suppose that the traces $c_k:=u_k|_{\partial B_1}$ are such that $\|c_k\|_{L^2(\partial B_1)}=1$. Notice that 
as in \cref{prop:blow-up-homo2}, we have that $u_k$ converges in $C^{1,\alpha}(B_1^+)$, up to subsequences, to some function $p$ which is a $(2m+1)$-homogeneous solution. In particular, $p\in\mathcal{P}_{2m+1}$ and $\|p\|_{L^2(\partial B_1)}=1.$ This means that $$\|u_k-p\|_{L^\infty(B_{3/2})}\le \eta_3 \quad\text{for every}\quad k>k_0,$$
for some $k_0\in\N$, where $\eta_3>0$ is defined in \cref{lemma:B.3}.
Therefore $$u_k\equiv 0\quad\text{in}\quad \mathcal{Z}_\delta\quad\text{for every}\quad k>k_0,$$ by \cref{lemma:B.3}. 
Moreover we can suppose that $$|W_{2m+1}(u_k)|\le \eta\quad\text{for every}\quad k>k_0,$$ which follows by \cref{lemma:gap} with $\eta>0$ as in \eqref{eq:lowerbound}. Then the function $u_k$ satisfies the hypotheses of \cref{thm:epi} and \cref{prop:epi-negative}.

Passing to a subsequence, we can suppose that either $t_k>0$ for every $k>k_0$ or $t_k<0$ for every $k>k_0$. In the first case we use \cref{thm:epi}, while
in the second case we use \cref{prop:epi-negative}. For simplicity, we suppose that $t_k<0$ for every $k>k_0$, the other case being analogous.
By \cref{lemma:gap} \be\label{eq:eq13}W_{2m+1}(r^{{2m+1}+t_k}c_k)=t_k \|c_k\|_{L^2(\partial B_1)}^2=t_k<0\ee and $$W_{2m+1}(r^{2m+1}c)=\left(1+C_mt_k\right) t_k, \quad\text{where}\quad C_m=\frac{1}{n+2(2m+1)-1}.$$
Then, by the epiperimetric inequality in \cref{prop:epi-negative}, we have that for every $k>k_0$
$$\begin{aligned}W_{2m+1}(r^{{2m+1+t_k}}c_k)&\le \left(1+ |(1+C_mt_k)t_k|\right)W_{2m+1}(r^{{2m+1}}c_k)\\&=\left(1- (1+C_mt_k)t_k\right)\left(1+C_mt_k\right) W_{2m+1}(r^{{2m+1}+t_k}c_k),
			\end{aligned}$$ where in the last equality we used \cref{lemma:gap}.
			Then by \eqref{eq:eq13}
			$$\left(1- (1+C_mt_k)t_k\right)\left(1+C_mt_k\right)\le 1\quad\text{for every}\quad k>k_0,$$ which implies that $$-t_k+C_mt_k+O(t_k^2)\le 0\quad\text{for every}\quad k>k_0,$$
			which is a contradiction by the definition of $C_m$ and the fact that $t_k\to0^-$. 
\end{proof}

\bibliographystyle{alpha}
\bibliography{thin-obstacle-references.bib}

\end{document}